\documentclass{amsart}
\usepackage{fullpage}
\usepackage{stmaryrd}
\usepackage{amssymb}
\usepackage{mathrsfs}

\newtheorem{theorem}{Theorem}[section]
\newtheorem{lemma}[theorem]{Lemma}

\newtheorem{prop}[theorem]{Proposition}

\newtheorem*{thmA}{Theorem A}
\newtheorem*{thmB}{Theorem B}

\theoremstyle{definition}

\newtheorem{hyp}[theorem]{Hypothesis}

\theoremstyle{remark}

\numberwithin{equation}{section}

\newcommand{\bC}{{\mathbf{C}}}

\newcommand{\bF}{{\mathbf{F}}}

\newcommand{\Aut}{{\operatorname{Aut}}}
\newcommand{\Stab}{{\operatorname{Stab}}}
\newcommand{\Irr}{{\operatorname{Irr}}}

\newcommand{\Syl}{{\operatorname{Syl}}}
\newcommand{\Sym}{{\operatorname{Sym}}}

\newcommand{\dl}{{\operatorname{dl}}}
\newcommand{\cl}{{\operatorname{cl}}}

\newcommand{\Out}{{\operatorname{Out}}}
\newcommand{\Alt}{{\operatorname{Alt}}}

\let\nor=\triangleleft

\begin{document}

\title{On $p$-parts of character degrees and conjugacy class sizes of finite groups}

\author{Yong Yang and Guohua Qian}
\address{Department of Mathematics, Texas State University, 601 University Drive, San Marcos, TX 78666, USA.}
\address{Key Laboratory of Group and Graph Theories and Applications, Chongqing University of Arts and Sciences, Chongqing, China.}
\address{Department of Mathematics, Changshu Institute of Technology, Changshu, JiangSu 215500, Peoples Republic of China.}
\makeatletter
\email{yang@txstate.edu, ghqian@cslg.cn}
\makeatother

\date{}



\begin{abstract}
Let $G$ be a finite group and $\Irr(G)$ the set of irreducible complex characters of $G$. Let $e_p(G)$ be the largest integer such that $p^{e_p(G)}$ divides $\chi(1)$ for some $\chi \in \Irr(G)$. We show that $|G:\bF(G)|_p \leq p^{k e_p(G)}$ for a constant $k$. This settles a conjecture of A. Moret\'o ~\cite[Conjecture 4]{Moret1}. 


We also study the related problems of the $p$-parts of conjugacy class sizes of finite groups.
\end{abstract}

\maketitle
\maketitle
\Large
\section{Introduction} \label{sec:introduction8}

Let $G$ be a finite group and $P$ be a Sylow $p$-subgroup of $G$, it is reasonable to expect that the degrees of irreducible characters of $G$ somehow restrict those of $P$.  The Ito-Michler theorem proves that each ordinary irreducible character degree is coprime to $p$ if and only if $G$ has a normal abelian Sylow $p$-subgroup. Of course, this implies that $|G:\bF(G)|_p=1$.



Let $\Irr(G)$ be the set of irreducible complex characters of $G$, and $e_p(G)$ be the largest integer such that $p^{e_p(G)}$ divides $\chi(1)$ for some $\chi \in \Irr(G)$. Isaacs ~\cite{Isaacs} showed that if $G$ is solvable, then the derived length of a Sylow $p$-subgroup of $G$ is bounded above by $2e_p(G) + 1$.

Let $b(P)$ denote the largest degree of an irreducible character of $P$. \cite[Conjecture 4]{Moret1} suggested that $\log b(P)$ is bounded as a function of $e_p(G)$. Moret\'o and Wolf ~\cite{MOWOLF} have proven this for $G$ solvable and even something a bit stronger, namely the logarithm to the base of $p$ of the $p$-part of $|G: \bF(G)|$ is bounded in terms of $e_p(G)$. In fact, they showed that $|G:\bF(G)|_p \leq p^{19 e_p(G)}$ for any solvable groups ~\cite[Corollary B (i)]{MOWOLF}, and $|G: \bF(G)|_p \leq p^{2 e_p(G)}$ for odd order groups ~\cite[Corollary B (iii)]{MOWOLF}. 

Recently, Lewis, Navarro and Wolf ~\cite{LNW} showed the following. Assume that $G$ is solvable and $e_p(G)=1$, then $|G:\bF(G)|_p \leq p^2$.

In this paper, we show that for arbitrary group, $|G:\bF(G)|_p \leq p^{k e_p(G)}$ for some constant $k$. This settles \cite[Conjecture 4]{Moret1}.

\begin{thmA}
Let $G$ be a finite group and $e_p(G)$ be the largest integer such that $p^{e_p(G)}$ divides $\chi(1)$ for some $\chi \in \Irr(G)$.
\begin{enumerate}
\item If $p \geq 5$, then $\log_p |G: \bF(G)|_p\leq 7.5 e_p(G)$.
\item If $p=3$, then $\log_p |G: \bF(G)|_p\leq 23 e_p(G)$.
\item If $p=2$, then $\log_p |G: \bF(G)|_p\leq 20 e_p(G)$.
\end{enumerate}
\end{thmA}

\bigskip

Let $G$ be a finite group. Let $p$ be a prime and we denote the $p$-part of the group order $|G|_p=p^n$. An irreducible ordinary character of $G$ is called $p$-defect $0$ if and only if its degree is divisible by $p^n$. By ~\cite[Theorem 4.18]{WF}, $G$ has a character of $p$-defect $0$ if and only if $G$ has a $p$-block of defect $0$.

It is an interesting problem to give necessary and sufficient conditions for the existence of $p$-blocks of defect zero. If a finite group $G$ has a character of $p$-defect $0$, then $O_p(G)=1$ ~\cite[Corollary 6.9]{WF}. Unfortunately, the converse is not true. Zhang ~\cite{Zhang} and Fukushima ~\cite{Hiroshi1, Hiroshi2} provided various sufficient conditions where a finite group $G$ has a block of defect zero.

Although the block of defect zero may not exist in general, one could try to find the smallest defect $d(B)$ of a block $B$ of $G$. It is not true in general that there exists a block $B$ with $d(B) \leq \lfloor \frac n 2 \rfloor$, as $G = A_7 (p = 2)$ shows us. However, the counterexamples were only found for $p=2$ and $p=3$. By work of Michler and Willems ~\cite{Michler,Willems} every simple group except possibly the alternating group has a block of defect zero for $p \geq 5$. The alternating group case was settled by Granville and Ono in ~\cite{GranvilleOno} using number theory.

For solvable groups, one of the results along this line is given by ~\cite[Theorem A]{AENA2}. In ~\cite{AENA2}, Espuelas and Navarro bounded the smallest defect $d(B)$ of a block $B$ of $G$ using the $p$-part of $|G|$. Using an orbit theorem ~\cite[Theorem 3.1]{AE1} of solvable linear groups of odd order, they showed the following result. Let $G$ be a (solvable) group of odd order such that $O_p(G) = 1$ and $|G|_p = p^n$, then $G$ contains a $p$-block $B$ such that $d(B) \leq \lfloor \frac n 2 \rfloor$. The bound is best possible, as shown by an example in ~\cite{AENA2}.

Based on these, the following question raised by Espuelas and Navarro in ~\cite{AENA2} seems to be natural. If $G$ is a finite group with $O_p(G)=1$ for some prime $p \geq 5$, and denote $|G|_p=p^n$, does $G$ contain a block of defect less then or equal to $\lfloor \frac n 2 \rfloor$?


Not much has been done on this problem. In a recent paper ~\cite[Theorem B]{YY5}, the author obtained a partial result toward the previous question by showing the following. Let $G$ be a finite solvable group, let $p$ be a prime such that $p \geq 5$ and $O_p(G)=1$, and we denote $|G|_p=p^n$. Then $G$ contains a $p$-block $B$ such that $d(B) \leq \lfloor \frac {3n} {5} \rfloor $.

In this paper, we study a related result about the size of defect group of arbitrary finite groups, and obtain an upper bound.

\begin{thmB}
Let $p$ be a prime, and $G$ be a finite group such that $O_p(G)=1$. We denote the $p$ part of the group order $|G|_p=p^n$. Then $G$ contains a $p$-block $B$ such that $d(B) \leq \lfloor \alpha n \rfloor $ for some constant $\alpha$.
\end{thmB}

\section{An orbit theorem} \label{sec:Orbits}
Theorem A and Theorem B are proved using an orbit theorem of $p$-solvable groups.

We first fix some notation:
\begin{enumerate}


\item We use $\bF(G)$ to denote the Fitting subgroup of $G$. Let $\bF_0(G) \leq \bF_1(G) \leq \bF_2(G) \leq \cdots \leq \bF_n(G)=G$ denote the ascending Fitting series, i.e. $\bF_0(G)=1$, $\bF_1(G)=\bF(G)$ and $\bF_{i+1}(G)/\bF_i(G)=\bF(G/\bF_i(G))$. $\bF_i(G)$ is the $i$th ascending Fitting subgroup of $G$.

\item We use $F^*(G)$ to denote the generalized Fitting subgroup of $G$.

\item We use $O_{\infty}(G)$ to denote the maximal normal solvable subgroup of $G$.

\item We use $\Irr(G)$ to denote the set of all the irreducible characters of the group $G$.

\item We use $\cl(G)$ to denote the set of all the conjugacy classes of the group $G$.

\item Let $G$ be a finite group, we denote $cd(G)=\{\chi(1)\ | \ \chi \in \Irr(G) \}$.

\item Let $G$ be a finite solvable group, we denote $\dl(G)$ to be the derived length of $G$.

\end{enumerate}

\bigskip

We need the following results about simple groups.

\begin{lemma}\label{simplecoprime}
Let $A$ act faithfully and coprimely on a nonabelian simple group $S$. Then $A$ has at least $2$ regular orbits on $\Irr(S)$.
\end{lemma}
\begin{proof}
This is ~\cite[Proposition 2.6]{MOTIEP}.
\end{proof}


\begin{lemma}\label{chardegsimple}
Let $G$ be a non-abelian finite simple group, then $|cd(G)| \geq 4$.
\end{lemma}
\begin{proof}
This follows from ~\cite[Theorem 12.15]{Isaacs/book}.
\end{proof}

\begin{lemma}\label{simpleouter}
Let $G$ be a non-abelian finite simple group, then $C=\Out(G)$ has a normal series of the form $A \nor B \nor C$ where $A$ is abelian, $B/A$ is cyclic and $C/B \cong 1, S_2$ or $S_3$.
\end{lemma}
\begin{proof}
This is a standard result by the classification of finite simple groups.
\end{proof}


We now prove an orbit theorem about $p$-solvable groups.

\begin{prop}\label{prop3}
Let $p$ be a prime such that $p \geq 3$, and assume $H$ is a $p$-solvable group such that $\bF(H) =1$. 
Suppose that $V=F^*(H) = S_1 \times \cdots \times S_n$ where $S_i \cong S$ for $1 \le k \le n$ and $S$ is a nonabelian simple group. Then $H$ acts faithfully on $V$. Since $V$ acts by inner automorphisms on $V$ this implies that $G=H/V$ embeds in $\Out(V) \cong \Out(S) \wr \Sym(n)$. Moreover $G$ acts on $\Irr(V)$. We consider the action of $G$ on $\Irr(V)$ and have the followings,

\begin{enumerate}
\item There exist four $G$-orbits with representatives $v_1$, $v_2$, $v_3$ and $v_4 \in \Irr(V)$ of different degrees such that for any $P \in \Syl_p(G)$, we have $\bC_P(v_j) \subseteq \Out(S_1) \times \cdots \times \Out(S_n)$ for $1 \leq j \leq 4$.

\item Assume $p \geq 5$, then there exists $N \nor G$,  $N \subseteq \bF_2(G)$ and there exist four $G$-orbits with representatives $v_1$, $v_2$, $v_3$ and $v_4 \in \Irr(V)$ of different degrees such that for any $P \in \Syl_p(G)$, we have $\bC_P(v_j) \subseteq N$ for $1 \leq j \leq 4$. Moreover, the Sylow $p$-subgroup of $N \bF(G)/\bF(G)$ is abelian. 

\item Assume $p = 3$, then there exists $N \nor G$,  $N \subseteq \bF_3(G)$ and there exist four $G$-orbits with representatives $v_1$, $v_2$, $v_3$ and $v_4 \in \Irr(V)$ of different degrees such that for any $P \in \Syl_p(G)$, we have $\bC_P(v_j) \subseteq N$ for $1 \leq j \leq 4$. Moreover, the Sylow $p$-subgroup of $N \bF_2(G)/\bF_2(G)$ is abelian.
\end{enumerate}

\end{prop}
\begin{proof}

(2) and (3) follow from (1) since $C=\Out(S)$ has a normal series of the form $A \nor B \nor C$ where $A$ is abelian, $B/A$ is cyclic and $C/B \cong 1, S_2$ or $S_3$ by Lemma ~\ref{simpleouter}. For (2) we choose $N=\bF_2(G \cap (\Out(S_1) \times \cdots \times \Out(S_n)))$. For (3) we choose $N=\bF_3(G \cap (\Out(S_1) \times \cdots \times \Out(S_n)))$.

Thus, we only need to show (1).

We note that $\Irr(V)=\Irr(S_1) \times \cdots \times \Irr(S_n)$. We denote $\bar G=G/(G \cap (\Out(S_1) \times \cdots \times \Out(S_n)))$, and clearly $\bar G$ is a permutation group on a set of $n$ elements.

Step 1. Assume the action of $\bar G$ is not transitive, then the result follows easily by induction.


We recall some basic facts about the decompositions of transitive groups. Let $\bar G$ be a transitive permutation group acting on a set $\Delta$, $|\Delta| = n$. A system of imprimitivity is a partition of $\Delta$, invariant under $\bar G$. A primitive group has no non-trivial system of imprimitivity. Let $(\Delta_1, \cdots, \Delta_m)$ denote a system of imprimitivity of $\bar G$ with maximal block-size $b$ $(1 \leq b <n; b = 1$ if and only if $\bar G$ is primitive; $bm=n$).

Step 2. Assume $n>1$ and the action of $\bar G$ is transitive, then either $\bar G$ is imprimitive and $b>1$ or $\bar G$ itself is primitive and $b=1$.

Let $\bar J_i$ be the set-wise stabilizer of $\bar G$ on $\Delta_i$, i.e. $\bar J_i=\Stab_{\bar G}(\Delta_i)$. Then the groups $\bar J_i$s are permutationally equivalent transitive groups of degree $b$. Let $\bar K =\cap_i \bar J_i$. We know that $\bar K$ is a normal subgroup of $\bar G$ stabilizing each of the blocks $\Delta_i$, and $\bar G/\bar K$ is a primitive group of degree $m$ acting upon the set of blocks $\Delta_i$ and hence $\{1,\cdots, m\}$. We denote the set $\{1,\cdots, m\}$ to be $\Omega$.




  We define $V_i=\prod_{t \in \Delta_i} S_t$. Consider the action of $\bar J_i$ on $\Irr(V_i)=\prod_{t \in \Delta_i} \Irr(S_t)$. If $\bar G$ is imprimitive, then since $1< b <n$, by induction there exist four $\bar J_i$-orbits with representatives $\theta_i, \lambda_i, \chi_i$ and $\psi_i \in \Irr(V_i)$ such that for any $\bar P \in \Syl_p(\bar J_i)$, $\bC_{\bar P}(\theta_i) = \bC_{\bar P}(\lambda_i) = \bC_{\bar P}(\chi_i) = \bC_{\bar P}(\psi_i) =1 $. If $\bar G$ is primitive, then $b=1$ and the same conclusion holds by Lemma ~\ref{chardegsimple}.

  It is clear that we may choose $\theta_i$, $1 \leq i \leq m$ to be conjugate by the action of $\bar G/ \bar K$ and we can do the same for $\lambda_i$, $\chi_i$ and $\psi_i$, $1 \leq i \leq m$. We also denote $\theta(1)=\theta_i(1)$, $\lambda(1)=\lambda_i(1)$, $\chi(1)=\chi_i(1)$ and $\psi(1)=\psi_i(1)$. By re-indexing, we may assume that $\theta(1) > \lambda(1) > \chi(1) > \psi(1)$. 

 Assume $p \nmid |\bar G/\bar K|$, we set $v_1=\prod_{i=1}^m \theta_i, v_2=\prod_{i=1}^m \lambda_i, v_3=\prod_{i=1}^m \chi_i$ and $v_4=\prod_{i=1}^m \psi_i$. It is clear that $v_1(1)=\theta(1)^m > v_2(1)=\lambda(1)^m > v_3(1)=\chi(1)^m > v_4(1)=\psi(1)^m$. 

 Assume $p \mid |\bar G/\bar K|$ and $m \geq 5$, then since $\bar G/\bar K$ is $p$-solvable, we know that $\Alt(m) \not \leq \bar G/\bar K$. By \cite[Lemma 1]{SDperm}(b), we can find a partition $\Omega=\Omega_1 \cup \Omega_2 \cup \Omega_3 \cup \Omega_4$ such that $\Stab_{\bar G/ \bar K}(\Omega_1) \cap \Stab_{\bar G/ \bar K}(\Omega_2) \cap \Stab_{\bar G/ \bar K}(\Omega_3) \cap \Stab_{\bar G/ \bar K}(\Omega_4)$ is a $2$-group, and $|\Omega_1|$, $|\Omega_2|$, $|\Omega_3|$ and $|\Omega_4|$ are not all the same. We denote $t_i=|\Omega_i|$, $1 \leq i \leq 4$. By re-indexing, we may assume that $t_1 \geq t_2 \geq t_3 \geq t_4$.

 Since we know that not all the $t_i$s are the same, we must have one of the followings,
 \begin{enumerate}
 \item $t_1 > t_2 \geq t_3 \geq t_4$. In this case, we construct four irreducible characters,
  \begin{enumerate}
   \item $v_1=\prod_{i=1}^{m} \alpha_i$, where $\alpha_i=\theta_i$ if $i \in \Omega_1$, $\alpha_i=\lambda_i$ if $i \in \Omega_2$, $\alpha_i=\chi_i$ if $i \in \Omega_3$, $\alpha_i=\psi_i$ if $i \in \Omega_4$. 
   \item $v_2=\prod_{i=1}^{m} \beta_i$, where $\beta_i=\lambda_i$ if $i \in \Omega_1$, $\beta_i=\theta_i$ if $i \in \Omega_2$, $\beta_i=\chi_i$ if $i \in \Omega_3$, $\beta_i=\psi_i$ if $i \in \Omega_4$. 
   \item $v_3=\prod_{i=1}^{m} \gamma_i$, where $\gamma_i=\chi_i$ if $i \in \Omega_1$, $ \gamma_i=\theta_i$ if $i \in \Omega_2$, $ \gamma_i=\lambda_i$ if $i \in \Omega_3$, $ \gamma_i=\psi_i$ if $i \in \Omega_4$. 
   \item $v_4=\prod_{i=1}^{m} \delta_i$, where $\delta_i=\psi_i$ if $i \in \Omega_1$, $\delta_i=\theta_i$ if $i \in \Omega_2$, $\delta_i=\lambda_i$ if $i \in \Omega_3$, $\delta_i=\chi_i$ if $i \in \Omega_4$. 
  \end{enumerate}
 Those four characters have different degrees since

 $v_1(1)=\theta(1)^{t_1} \lambda(1)^{t_2} \chi(1)^{t_3} \psi(1)^{t_4} > v_2(1)=\lambda(1)^{t_1} \theta(1)^{t_2} \chi(1)^{t_3} \psi(1)^{t_4} >$\\
 $v_3(1)= \chi(1)^{t_1} \theta(1)^{t_2} \lambda(1)^{t_3} \psi(1)^{t_4}> v_4(1)=\psi(1)^{t_1} \theta(1)^{t_2} \lambda(1)^{t_3} \chi(1)^{t_4}$.

 \item $t_1=t_2>t_3 \geq t_4$. In this case, we construct four irreducible characters,
  \begin{enumerate}
   \item $v_1=\prod_{i=1}^{m} \alpha_i$, where $\alpha_i=\theta_i$ if $i \in \Omega_1$, $\alpha_i=\lambda_i$ if $i \in \Omega_2$, $\alpha_i=\chi_i$ if $i \in \Omega_3$, $\alpha_i=\psi_i$ if $i \in \Omega_4$. 
   \item $v_2=\prod_{i=1}^{m} \beta_i$, where $\beta_i=\theta_i$ if $i \in \Omega_1$, $\beta_i=\chi_i$ if $i \in \Omega_2$, $\beta_i=\lambda_i$ if $i \in \Omega_3$, $\beta_i=\psi_i$ if $i \in \Omega_4$. 
   \item $v_3=\prod_{i=1}^{m} \gamma_i$, where $\gamma_i=\theta_i$ if $i \in \Omega_1$, $\gamma_i=\psi_i$ if $i \in \Omega_2$, $\gamma_i=\lambda_i$ if $i \in \Omega_3$, $\gamma_i=\chi_i$ if $i \in \Omega_4$. 
   \item $v_4=\prod_{i=1}^{m} \delta_i$, where $\delta_i=\chi_i$ if $i \in \Omega_1$, $\delta_i=\psi_i$ if $i \in \Omega_2$, $\delta_i=\lambda_i$ if $i \in \Omega_3$, $\delta_i=\theta_i$ if $i \in \Omega_4$. 
  \end{enumerate}
 Those four characters have different degrees since

 $v_1(1)=\theta(1)^{t_1} \lambda(1)^{t_2} \chi(1)^{t_3} \psi(1)^{t_4} > v_2(1)=\theta(1)^{t_1} \chi(1)^{t_2} \lambda(1)^{t_3} \psi(1)^{t_4} >$\\
 $v_3(1)=\theta(1)^{t_1} \psi(1)^{t_2} \lambda(1)^{t_3} \chi(1)^{t_4} > v_4(1)= \chi(1)^{t_1} \psi(1)^{t_2} \lambda(1)^{t_3} \theta(1)^{t_4}$.

 \item $t_1=t_2=t_3 > t_4$. In this case, we construct four irreducible characters,
   \begin{enumerate}
     \item $v_1=\prod_{i=1}^{m} \alpha_i$, where $\alpha_i=\theta_i$ if $i \in \Omega_1$, $\alpha_i=\lambda_i$ if $i \in \Omega_2$, $\alpha_i=\chi_i$ if $i \in \Omega_3$, $\alpha_i=\psi_i$ if $i \in \Omega_4$. 
     \item $v_2=\prod_{i=1}^{m} \beta_i$, where $\beta_i=\theta_i$ if $i \in \Omega_1$, $\beta_i=\lambda_i$ if $i \in \Omega_2$, $\beta_i=\psi_i$ if $i \in \Omega_3$, $\beta_i=\chi_i$ if $i \in \Omega_4$. 
     \item $v_3=\prod_{i=1}^{m} \gamma_i$, where $\gamma_i=\theta_i$ if $i \in \Omega_1$, $\gamma_i=\chi_i$ if $i \in \Omega_2$, $\gamma_i=\psi_i$ if $i \in \Omega_3$, $\gamma_i=\lambda_i$ if $i \in \Omega_4$. 
     \item $v_4=\prod_{i=1}^{m} \delta_i$, where $\delta_i=\lambda_i$ if $i \in \Omega_1$, $\delta_i=\chi_i$ if $i \in \Omega_2$, $\delta_i=\psi_i$ if $i \in \Omega_3$, $\delta_i=\theta_i$ if $i \in \Omega_4$. 
   \end{enumerate}
 Those four characters have different degrees since

$v_1(1)=\theta(1)^{t_1} \lambda(1)^{t_2} \chi(1)^{t_3} \psi(1)^{t_4}> v_2(1)=\theta(1)^{t_1} \lambda(1)^{t_2} \psi(1)^{t_3} \chi(1)^{t_4}>$ \\
$v_3(1)=\theta(1)^{t_1} \chi(1)^{t_2} \psi(1)^{t_3} \lambda(1)^{t_4} > v_4(1)= \lambda(1)^{t_1} \chi(1)^{t_2} \psi(1)^{t_3} \theta(1)^{t_4}$.
\end{enumerate}

Assume $p \mid |\bar G/ \bar K|$, $m=4$ and $p=3$. We set $v_1=\theta_1 \lambda_2  \chi_3 \psi_4$, $v_2=\theta_1  \lambda_2  \psi_3 \psi_4$, $v_3=\theta_1 \chi_2 \psi_3 \psi_4$ and $v_4=\lambda_1 \chi_2 \psi_3 \psi_4$. They have different degrees since $v_1(1)=\theta(1) \lambda(1) \chi(1) \psi(1)> v_2(1)=\theta(1)  \lambda(1)  \psi^2(1) > v_3(1)=\theta(1) \chi(1) \psi^2(1) > v_4(1)= \lambda(1) \chi(1) \psi^2(1)$.

 It is clear that $\bC_{\bar G/ \bar K}(v_1),\bC_{\bar G/ \bar K}(v_2),\bC_{\bar G/ \bar K}(v_3)$ and $\bC_{\bar G/ \bar K}(v_4)$ is a $2$-group.\\

 Assume $p \mid |\bar G/ \bar K|$, $m=3$ and  $p=3$. We set $v_1=\theta_1 \lambda_2  \chi_3$, $v_2=\theta_1  \lambda_2  \psi_3$, $v_3=\theta_1 \chi_2 \psi_3$ and $v_4=\lambda_1 \chi_2 \psi_3$. They have different degrees since $v_1(1)=\theta(1) \lambda(1)  \chi(1)> v_2(1)=\theta(1)  \lambda(1)  \psi(1)> v_3(1)=\theta(1) \chi(1) \psi(1) > v_4(1)= \lambda(1) \chi(1) \psi(1)$.

 It is clear that $\bC_{\bar G/ \bar K}(v_1),\bC_{\bar G/ \bar K}(v_2),\bC_{\bar G/ \bar K}(v_3)$ and $\bC_{\bar G/ \bar K}(v_4)$ is a $2$-group.\\

Thus, we may always find four irreducible characters $v_1$, $v_2$, $v_3$ and $v_4$ in $\Irr(V)$ of different degrees such that for any $P \in \Syl_p(G)$, we have $\bC_P(v_j) \subseteq K$ and thus $\bC_P(v_j) \subseteq  \Out(S)^n$ for $1 \leq j \leq 4$.


Step 3. Assume $n=1$, then the result follows by Lemma ~\ref{chardegsimple}.
\end{proof}

\begin{prop}\label{prop4}
Let $G$ be a finite $p$-solvable group where $O_{\infty}(G)=1$. Then $F^*(G)=E_1 \times \dots \times E_m$ is a product of $m$ finite non-abelian simple groups $E_j$, $1 \leq j \leq m$ permuted by $G$. We denote $\bar G=G/F^*(G)$. 
\begin{enumerate}
\item Assume $p \ge 5$, then there exists $N \nor \bar G$ where $N \subseteq \bF_2(\bar G)$, and there exists $v \in \Irr(F^*(G))$ such that for any $P \in \Syl_p(\bar G)$, we have $\bC_P(v) \subseteq N$. Moreover, the Sylow $p$-subgroup of $N \bF(\bar G)/\bF(\bar G)$ is abelian.

\item Assume $p = 3$, then there exists $N \nor \bar G$ where $N \subseteq \bF_3(\bar G)$, and there exists $v \in \Irr(F^*(G))$ such that for any $P \in \Syl_p(\bar G)$, we have $\bC_P(v) \subseteq N$. Moreover, the Sylow $p$-subgroup of $N \bF_2(\bar G)/\bF_2(\bar G)$ is abelian.
\end{enumerate}
\end{prop}
\begin{proof}


Clearly $G$ acts faithfully on $F^*(G)$. Since $F^*(G)$ acts by inner automorphisms on $F^*(G)$, this implies that $\bar G$ embeds in $\Out(F^*(G))$. Moreover $\bar G$ acts on $\Irr(F^*(G))$. Next, we group the simple groups in the direct product of $F^*(G)$ where $G$ acts transitively. We denote $F^*(G)=L_1 \times \cdots \times L_s$, $L_i=E_{i1} \times E_{i2} \times \cdots \times E_{im_i}$, $1 \leq i \leq s$ where $G$ transitively permutes the simples groups inside the direct product of $L_i$. Clearly $E_{i1} \cong E_{i2} \cdots \cong E_{im_i}$.

We see that $\bar G$ can be embedded as a subgroup of $\Out(L_1) \times \cdots \times \Out(L_s)$ and we denote $K_i$ to be the image of $\bar G$ in $\Out(L_i)$.


If $p \geq 5$, by applying Proposition \ref{prop3}(1) to the action of $K_i$ on $\Irr(L_i)$, there exists $v_{i} \in \Irr(L_i)$, and $N_i \nor K_i$ such that for any $P_i \in \Syl_{p}(K_i)$, $\bC_{P_i}(v_{i}) \subseteq N_i \subseteq \bF_2(K_i)$. Also the Sylow $p$-subgroup of $N_i \bF(K_i)/\bF(K_i)$ is abelian. 

Let $v=\sum v_{i}$ and $N= K \cap (\prod N_i)$. Let $P \in \Syl_{p}(G)$, and $P_i$ to be the image of $P$ in $\Irr(V_i)$. Then $\bC_P(v) \subseteq \prod \bC_{P_i}(v_{i}) \subseteq \prod N_i$. Clearly $N \bF(K)/\bF(K) \subseteq \prod N_i \bF(K_i)/\bF(K_i)$ and the result follows.

If $p=3$, by applying Proposition \ref{prop3}(2) to the action of $K_i$ on $\Irr(L_i)$, there exists $v_{i} \in \Irr(L_i)$, and $N_i \nor K_i$ such that for any $P_i \in \Syl_{p}(K_i)$, $\bC_{P_i}(v_{i}) \subseteq N_i \subseteq \bF_3(K_i)$. Also the Sylow $p$-subgroup of $N_i \bF_2(K_i)/\bF_2(K_i)$ is abelian. 

Let $v=\sum v_{i}$ and $N= K \cap (\prod N_i)$. Let $P \in \Syl_{p}(G)$, and $P_i$ to be the image of $P$ in $\Irr(V_i)$. Then $\bC_P(v) \subseteq \prod \bC_{P_i}(v_{i}) \subseteq \prod N_i$. Clearly $N \bF_3(K)/\bF_3(K) \subseteq \prod N_i \bF_3(K_i)/\bF_3(K_i)$ and the result follows.
\end{proof}

\section{$p$ part of $|G:\bF(G)|$, character degrees and conjugacy class sizes} \label{p part of G/F(G)}

We now prove Theorem A and some related results in this section. We first obtain bounds for $p$-solvable groups and then extend those to arbitrary groups.\\



\begin{theorem}\label{chardegreepsolvablebound}
Let $G$ be a finite $p$-solvable group where $p$ is a prime and $p \geq 5$. Suppose that $p^{a+1}$ does not divide $\chi(1)$ for all $\chi \in \Irr(G)$ and let $P \in \Syl_p(G)$, then $|G: \bF(G)|_p\leq p^{5.5a}$, $b(P)\leq p^{6.5a}$ and $\dl(P) \leq 5 + \log_2 a + \log_2 6.5$.
\end{theorem}
\begin{proof}
Let $T=O_{\infty}(G)$, the maximal normal subgroup of $G$. Since $\bF(G) \subseteq T$, $\bF(T)=\bF(G)$. Since $T \nor G$, $p^{a+1}$ does not divide $\lambda(1)$ for all $\lambda \in \Irr(T)$. Thus by ~\cite[Remark of Corollary 5.3]{YY5}, $|T: \bF(G)|_p\leq p^{2.5a}$.

Let $\tilde G=G/T$ and $\bar G=\tilde G / F^*(\tilde G)$. It is clear that $F^*(\tilde G)$ is a direct product of finite non-abelian simple groups. By Proposition ~\ref{prop4}(1), there exists $v \in \Irr(F^*(\tilde G))$, $N \nor \bar G$ where $N \subseteq \bF_2(\bar G)$ such that for any $P \in \Syl_p(\bar G)$, we have $\bC_P(v) \subseteq N$, and the Sylow $p$-subgroup of $N \bF(\bar G)/\bF(\bar G)$ is abelian. It is clear that we may find $\tilde \gamma \in \Irr(\tilde G)$ such that $|\bar G: N|_{p}$ divides $\tilde \gamma(1)$.

Let $P/\bF(\bar G)$ be a Sylow $p$-subgroup of $N \bF(\bar G)/\bF(\bar G)$. Where $Y = O_{p'}(\bF(\bar G))$, observe that $W = \Irr(Y/\Phi(Y))$ is a faithful and completely reducible $P/\bF(\bar G)$-module. By Gow's regular orbit theorem ~\cite[2.6]{GOW}, we have $\mu \in W$ such that $\bC_P(\mu) = \bF(\bar G)$. We may view $\mu$ as a character of the preimage $X$ of $Y$ in $\bar G$. Take $\bar \alpha \in \Irr(\bar G)$ lying over $\mu$, and $\bar \alpha$ lies over an irreducible character $\bar \psi$ of $P$ lying over $\mu$. Clearly, $\bar \psi(1)_p \ge |N \bF(\bar G)/\bF(\bar G)|_p$. As $P$ is normal in $G$, we have $\bar \alpha(1)_p \ge \bar \psi(1)_p \ge |N \bF(\bar G)/\bF(\bar G)|_p$.

Let $P_1$ be the Sylow $p$-subgroup of $\bF(\bar G) \cap N$. By Lemma ~\ref{simplecoprime}, we may find $\nu \in \Irr(F^*(\tilde G))$ such that $\bC_{P_1}(\nu)=1$. Thus by a similar argument as before, we may find an irreducible character $\tilde \beta$ of $\tilde G$ such that $\tilde \beta(1)_p \geq |N \cap \bF(\bar G)|_p$.

Thus $|G:\bF(G)|_p \leq p^{5.5a}$. If $P \in \Syl_p(G)$, then $b(P) \leq |P: O_p(G)||b(O_p(G))|=|G:\bF(G)|_p |b(O_p(G))| \leq p^{5.5a}p^a=p^{6.5a}$.

Now, we want to prove the last part of the statement. By ~\cite[Theorem 12.26]{Isaacs/book} and the nilpotency of $P$, we have that $P$ has an abelian subgroup $B$ of index at most $b(P)^4$. By ~\cite[Theorem 5.1]{Podoski}, we deduce that $P$ has a normal abelian subgroup $A$ of index at most $|P:B|^2$. Thus, $|P:A| \leq |P:B|^2 \leq b(P)^{8s}$, where $b(P)=p^s$. By ~\cite[Satz III.2.12]{Huppert1}, $\dl(P/A) \leq 1+\log_2(8s)$ and so $\dl(P) \leq 2+ \log_2(8s)=5+\log_2(s)$. Since $s$ is at most $6.5a$, the result follows.
\end{proof}


We now state the conjugacy analogs of Theorem ~\ref{chardegreepsolvablebound}. Given a group $G$, we write $b^*(G)$ to denote the largest size of the conjugacy classes of $G$. 
\begin{theorem}\label{conjugacyboundpsolvable}
Let $G$ be a $p$-solvable group where $p$ is a prime and $p \geq 5$. Suppose that $p^{a+1}$ does not divide $|C|$ for all $C \in \cl(G)$ and let $P \in \Syl_p(G)$, then $|G: \bF(G)|_p \leq p^{5.5a}$, $b^*(P)\leq p^{6.5a}$ and $|P'| \leq p^{6.5a(6.5a+1)/2}$.
\end{theorem}
\begin{proof}
The proof of the first statement goes similarly as the previous one. Write $N=O_p(G)$. It is clear that $|N: \bC_N(x)|$ divides $|G: \bC_G(x)|$ for all $x \in G$. Thus, if we take $x \in P$ we have that
\[|\cl_P(x)|= |P: \bC_P(x)| \leq |P:N||N: \bC_N(x)| \leq p^{5.5a} p^a=p^{6.5a}\]
 Finally, to obtain the bounds for the order of $P'$ is suffices to apply a theorem of Vaughan-Lee \cite[Theorem VIII.9.12]{Huppert2}.
\end{proof}

\begin{theorem}\label{correction}
If $G$ is solvable, there exists a product $\theta=\chi_1 \dots \chi_t$ of distinct irreducible characters $\chi_i$ of $G$ such that $|G:\bF(G)|$ divides $\theta(1)$ and $t \leq 15.$ Furthermore, if $|\bF_8(G)|$ is odd then we can take $t \leq 3$ and if $|G|$ is odd we can take $t \leq 2$.
\end{theorem}
\begin{proof}
This is ~\cite[Theorem 5.1]{YY1}. The statement in that paper should be $t \leq 15$ instead of $t \leq 19$, and we take the opportunity to correct it here.
\end{proof}

\begin{theorem}\label{correctionconj}
If $G$ is solvable, there exists conjugacy classes $C_1, \cdots, C_t$ such that $|G:\bF(G)|$ divides $|C_1| \cdots |C_t|$ and $t \leq 15 $. Furthermore, if $|\bF_8(G)|$ is odd then we can take $t \leq 3$ and if $|G|$ is odd we can take $t \leq 2$.
\end{theorem}
\begin{proof}
This is the conjugacy class version of Theorem ~\ref{correction}.
\end{proof}

For the case $p=3$, we have the following result.

\begin{theorem}\label{chardegreeboundnew3}
Let $G$ be a $p$-solvable group where $p = 3$. Suppose that $p^{a+1}$ does not divide $\chi(1)$ for all $\chi \in \Irr(G)$ and let $P \in \Syl_p(G)$, then $|G: \bF(G)|_p\leq p^{20a}$, $b(P)\leq p^{21a}$ and $\dl(P) \leq 5 + \log_2 a + \log_2 21$.
\end{theorem}
\begin{proof}
The proof is similar to the proof of Theorem ~\ref{chardegreepsolvablebound} but using Theorem ~\ref{correction} instead of ~\cite[Remark of Corollary 5.3]{YY5}, and Proposition ~\ref{prop4}(2) instead of Proposition ~\ref{prop4}(1).
\end{proof}

\begin{theorem}\label{conjugacybound3}
Let $G$ be a $p$-solvable group where $p = 3$. Suppose that $p^{a+1}$ does not divide $|C|$ for all $C \in \cl(G)$ and let $P \in \Syl_p(G)$, then $|G: \bF(G)|_p \leq p^{20a}$, $b^*(P)\leq p^{21a}$ and $|P'| \leq p^{21a(21a+1)/2}$.
\end{theorem}
\begin{proof}
This is the conjugacy class version of Theorem ~\ref{chardegreeboundnew3}.
\end{proof}

\begin{lemma} \label{simplepchar}
Let $S$ be a finite simple group and $p$ be a prime divisor of $|S|$.
\begin{enumerate}
\item If $p \geq 5$, then there exist $\chi \in \Irr(S)$ such that $v_p(|\Aut(S)|/\chi^2(1)) < 0$.
\item If $p=3$, then there exist $\chi \in \Irr(S)$ such that $v_p(|\Aut(S)|/\chi^3(1)) < 0$.
\item If $p=2$, then there exist $\chi \in \Irr(S)$ such that $v_p(|\Aut(S)|/\chi^5(1)) < 0$.
\end{enumerate}
\end{lemma}
\begin{proof}
This follows from ~\cite[Lemma 2.1]{Gagola}.
\end{proof}

\begin{lemma} \label{simplepconjugacy}
Let $S$ be a finite simple group and $p$ be a prime divisor of $|S|$.
\begin{enumerate}
\item If $p \geq 5$, then there exist $C \in \cl(S)$ such that $v_p(|\Aut(S)|/|C|^2) < 0$.
\item If $p=3$, then there exist $C \in \cl(S)$ such that $v_p(|\Aut(S)|/|C|^2) < 0$.
\item If $p=2$, then there exist $C \in \cl(S)$ such that $v_p(|\Aut(S)|/|C|^2) < 0$.
\end{enumerate}
\end{lemma}
\begin{proof}
For the simple group of Lie type and any prime $p$ or the alternating group and $p \geq 5$, there is a $p$-block of defect $0$. Hence there is a conjugacy class $\cl_G(x)$ of $p$-defect $0$, thus $|G|_p$ divides $|\cl_G(x)|$. And the result follows by the proof of ~\cite[Lemma 1.2]{Gagola}.

Thus one only needs to consider the alternating groups and $p=2,3$. 

Set $\alpha = (12 \cdots n)$ if $n$ is odd, we get that $\alpha \in A_n$ and $|\cl_{S_n}(\alpha)| = (n-1)!$. Thus $|\cl_{A_n}(\alpha)|$ is a multiple of $\frac 1 2  (n-1)!$ and the result is clear.

Set $\alpha = (12 \cdots n-1)$ if $n$ is even, we get that $\alpha \in A_n$ and $|\cl_{S_n}(\alpha)| =  \frac {n!} {n-1}$. Thus $|\cl_{A_n}(\alpha)|$ is a multiple of $\frac 1 2 \cdot \frac {n!} {n-1}$ and the result is clear.
\end{proof}

\begin{hyp} \label{hypothesis}
Let $p$ be a prime and let $N = W_1 \times \cdots \times W_s$ be a normal subgroup of a finite group $G$ with the following assumptions: $\bC_G(N) = 1$; every $W_i$, $1 \leq i \leq s$, is a nonabelian simple group of order divisible by $p$.
\end{hyp}

\begin{lemma}\label{qianinductionchar}
Let  $G$, $N$, $p$ be as in Hypothesis ~\ref{hypothesis}. If there exists $\phi_i \in \Irr(W_i)$ such that $v_p(|\Aut(W_i)|/ \phi_i(1)^k) < 0$ for every $i=1,\cdots,s$, then there exist $\phi \in \Irr(N)$ such that $v_p(|G|/\phi(1)^k) < 0$.
\end{lemma}
\begin{proof}
The proof is the same as ~\cite[Lemma 2.6]{QIAN}.
\end{proof}

\begin{lemma}\label{qianinductionconj}
Let  $G$, $N$, $p$ be as in Hypothesis ~\ref{hypothesis}. If there exists $C_i \in \cl(W_i)$ such that $v_p(|\Aut(W_i)|/ |C_i|^k) < 0$ for every $i=1,\cdots,s$, then there exist $C \in \cl(N)$ such that $v_p(|G|/|C|^k) < 0$.
\end{lemma}
\begin{proof}
The proof is the same as ~\cite[Lemma 2.6]{QIAN}.
\end{proof}

\begin{theorem}\label{chardegreegeneralbound}
Let $G$ be a finite group and $p$ be a prime. Suppose that $p^{a+1}$ does not divide $\chi(1)$ for all $\chi \in \Irr(G)$ and let $P \in \Syl_p(G)$. 
\begin{enumerate}
\item If $p \geq 5$, then $|G: \bF(G)|_p\leq p^{7.5a}$, $b(P)\leq p^{8.5a}$ and $\dl(P) \leq 5 + \log_2 a + \log_2 8.5$.
\item If $p=3$, then $|G: \bF(G)|_p\leq p^{23a}$, $b(P)\leq p^{24a}$ and $\dl(P) \leq 5 + \log_2 a + \log_2 24$.
\item If $p=2$, then $|G: \bF(G)|_p\leq p^{20a}$, $b(P)\leq p^{20a}$ and $\dl(P) \leq 5 + \log_2 a + \log_2 20$.
\end{enumerate}
\end{theorem}
\begin{proof}
Let $T$ be the maximal normal $p$-solvable subgroup of $G$. Since $\bF(G) \subseteq T$, $\bF(T)=\bF(G)$. Since $T \nor G$, $p^{a+1}$ does not divide $\lambda(1)$ for all $\lambda \in \Irr(T)$.

If $p \geq 5$, then $|T: \bF(G)|_p \leq p^{5.5a}$ by Theorem ~\ref{chardegreepsolvablebound}.

If $p = 3$, then $|T: \bF(G)|_p \leq p^{20a}$ by Theorem ~\ref{chardegreeboundnew3}.

If $p =2$, then $|T: \bF(G)|_p \leq p^{15a}$ by Theorem ~\ref{correction}.

We now consider $\bar G=G/T$, we know that $F^*(\bar G)$ is a direct product of the non-abelian simple groups, where $p$ divides the order of each of them.

Since $\bar G$ and $F^*(\bar G)$ satisfy Hypothesis ~\ref{hypothesis}, by Lemma ~\ref{qianinductionchar} and Lemma ~\ref{simplepchar}, we have that

$|\bar G|_p \leq p^{2a}$ if $p \geq 5$.

$|\bar G|_p \leq p^{3a}$ if $p=3$.

$|\bar G|_p \leq p^{5a}$ if $p=2$.

Thus, we have,
\begin{enumerate}
\item $|G:\bF(G)|_p \leq |G:T|_p |T: \bF(G)|_p  \leq p^{7.5a}$ if $p \geq 5$.
\item $|G:\bF(G)|_p \leq |G:T|_p |T: \bF(G)|_p  \leq p^{23 a}$ if $p=3$.
\item $|G:\bF(G)|_p \leq |G:T|_p |T: \bF(G)|_p  \leq p^{20 a}$ if $p=2$.
\end{enumerate}

The bounds for $b(P)$ and $\dl(P)$ follow from the same arguments as in Theorem ~\ref{chardegreepsolvablebound}.
\end{proof}

\begin{theorem}\label{conjugacygeneralbound}
Let $G$ be a finite group where $p$ is a prime. Suppose that $p^{a+1}$ does not divide $|C|$ for all $C \in \cl(G)$ and let $P \in \Syl_p(G)$. 
\begin{enumerate}
\item If $p \geq 5$, then $|G: \bF(G)|_p\leq p^{7.5a}$, $b^*(P)\leq p^{8.5a}$ and $|P'| \leq p^{8.5a(8.5a+1)/2}$.
\item If $p=3$, then $|G: \bF(G)|_p\leq p^{22a}$, $b^*(P)\leq p^{23a}$ and $|P'| \leq p^{23a(23a+1)/2}$.
\item If $p=2$, then $|G: \bF(G)|_p\leq p^{17a}$, $b^*(P)\leq p^{18a}$ and $|P'| \leq p^{18a(18a+1)/2}$.
\end{enumerate}
\end{theorem}
\begin{proof}
Let $T$ be the maximal normal $p$-solvable subgroup of $G$. Since $\bF(G) \subseteq T$, $\bF(T)=\bF(G)$. Since $T \nor G$, $p^{a+1}$ does not divide $|C|$ for all $C \in \cl(T)$.

If $p \geq 5$, then $|T: \bF(G)|_p \leq p^{5.5a}$ by Theorem ~\ref{conjugacyboundpsolvable}.

If $p=3$, then $|T: \bF(G)|_p \leq p^{20a}$ by Theorem ~\ref{conjugacybound3}.

If $p=2$, then $|T: \bF(G)|_p \leq p^{15a}$ by Theorem ~\ref{correctionconj}.

We now consider $\bar G=G/T$, we know that $F^*(\bar G)$ is a direct product of the non-abelian simple groups, where $p$ divides the order of each of them.

Since $\bar G$ and $F^*(\bar G)$ satisfy Hypothesis ~\ref{hypothesis}, by Lemma ~\ref{qianinductionconj} and Lemma ~\ref{simplepconjugacy}, we have that $|\bar G|_p \leq p^{2a}$.

Thus, we have,
\begin{enumerate}
\item $|G:\bF(G)|_p \leq |G:T|_p |T: \bF(G)|_p  \leq p^{7.5a}$ if $p \geq 5$.
\item $|G:\bF(G)|_p \leq |G:T|_p |T: \bF(G)|_p  \leq p^{22 a}$ if $p=3$.
\item $|G:\bF(G)|_p \leq |G:T|_p |T: \bF(G)|_p  \leq p^{17 a}$ if $p=2$.
\end{enumerate}

The bounds for $b^*(P)$ and $|P'|$ follow from the same arguments as in Theorem ~\ref{conjugacyboundpsolvable}.
\end{proof}

\section{Blocks and defect groups} \label{sec:Blocks}
Now we are ready to prove Theorem B, which we restate.\\

\begin{thmB} \label{sdefect}
Let $p$ be a prime, and $G$ be a finite group such that $O_p(G)=1$. We denote the $p$ part of the group order $|G|_p=p^n$. Then $G$ contains a $p$-block $B$ such that $d(B) \leq \lfloor \alpha n \rfloor $ where,
\begin{enumerate}
\item $\alpha = \frac {6.5} {7.5}$ if $p \geq 5$.
\item $\alpha = \frac {22} {23}$ if $p=3$.
\item $\alpha = \frac {19} {20}$ if $p=2$.
\end{enumerate}
\end{thmB}
\begin{proof}
By the proof of Theorem ~\ref{chardegreegeneralbound}, there exists a $\chi \in \Irr(G)$ such that $(\chi(1)_p)^{\beta} \geq |G|_p$, i.e. $\chi(1)_p \geq p^{\frac n {\beta}}$.

where
\begin{enumerate}
\item $\beta = 7.5$ if $p \geq 5$.
\item $\beta = 23$ if $p \geq 3$.
\item $\beta = 20$ if $p \geq 2$.
\end{enumerate}

If $|G|_p=p^n$ and $B$ be a $p$-block with $d(B)=d$, then $p^{n-d}$ is the largest $p$-part for all irreducible characters in $\Irr(G)\cap B$.

If $G$ has an irreducible character $\chi$ of degree divisible by $p^t$, then $d(B)\leq n-t$ where $B$ is a $p$-block which $\chi$ lies in.

Thus we know there is a  $p$-block such that $B$ such that $d(B) \leq n-\frac n {\beta} \leq \frac {(\beta-1)n}{\beta}$.
\end{proof}

Remark: In order to improve the bounds of the results in this paper, one might need to study the situation when a $p$-solvable group acts on a field of characteristic does not equal to $p$, and hope that a similar result as ~\cite[Theorem 3.3]{YY5} holds. Also, a strengthened version of Lemma ~\ref{simplecoprime} would also be helpful in improving the bounds.

\section{Acknowledgement} \label{sec:Acknowledgement}
The first author would like to thank for the financial support from the AMS-Simons travel grant. The second author is partially supported by a grant from the NSF of China (11471054).



\end{document}